\documentclass[11pt]{article}
\usepackage[T1]{fontenc}
\usepackage[margin=1in]{geometry}
\usepackage{mathptmx}
\usepackage{amsmath,amssymb,amsthm}
\usepackage{cite}
\usepackage{graphicx}
\usepackage{booktabs}
\usepackage{tikz}
\usepackage{pgfplots}
\pgfplotsset{compat=1.17}
\usepackage[hidelinks]{hyperref}

\newtheorem{theorem}{Theorem}[section]
\newtheorem{lemma}[theorem]{Lemma}
\newtheorem{proposition}[theorem]{Proposition}
\newtheorem{corollary}[theorem]{Corollary}
\theoremstyle{remark}
\newtheorem{remark}[theorem]{Remark}

\newcommand{\R}{\mathbf{R}}
\newcommand{\T}{\mathbf{T}}
\newcommand{\W}{\mathbf{W}}
\newcommand{\B}{\mathbf{B}}
\newcommand{\I}{\mathbf{I}}
\newcommand{\J}{\mathbf{J}}
\newcommand{\Q}{\mathbf{Q}}
\newcommand{\C}{\mathbf{C}}
\newcommand{\LM}{\mathbf{L}}
\newcommand{\PP}{\mathbf{P}}
\newcommand{\G}{\mathbf{G}}
\newcommand{\D}{\mathbf{D}}

\title{Exact fast factorizations of the AR(1) Karhunen--Lo{\`e}ve transform}

\author{Yuriy A.~Reznik\\[2pt]
\normalsize Massachusetts Institute of Technology, Cambridge, MA, USA\\
\normalsize \texttt{yreznik@mit.edu}}
\date{}

\begin{document}
\maketitle

\begin{abstract}
We derive a fast factorization of the exact Karhunen--Lo\`eve
transform (KLT) of an AR(1) source by mapping it onto the discrete
cosine transform. For even $N$ and every $\rho\in(0,1)$, the KLT
factors exactly into Chen's fast DCT-II structure---butterfly plus
fixed half-size DCT-II and DCT-IV cores---completed by two
orthogonal corrections: eigenvector matrices of
diagonal-plus-rank-one matrices carrying the entire
$\rho$-dependence. Fast DCT factorizations are reused
unchanged; as $\rho\to1$ the corrections become identities,
recovering Chen's algorithm. At short lengths the design is
explicit: at $N=4$ the transform is one butterfly and two
rotations with $\tan2\psi_s=-2/(1-\rho)$,
$\tan2\psi_a=2/(1+\rho)$; at $N=8$ every correction entry is in
radicals, via one quartic serving both branches, and factors into
six Givens rotations. Complete and scaled realizations run the
exact 8-point KLT in 32 and 24 multiplications, about twice the
fixed DCT-II; at $N=4$, a single extra multiplication.
At large $N$, fast-multipole application of the corrections yields
the exact KLT in $O(N\log N)$ operations to prescribed accuracy;
this overhead thus peaks at intermediate sizes and vanishes in
both limits.
\end{abstract}


\section{Introduction}
\label{sec:intro}

The KLT of the AR(1) (Markov-1) source is a classical benchmark in
transform coding. Ray and Driver~\cite{raydriver1970} solved its
eigenproblem up to a transcendental frequency equation. The discrete
cosine transform of Ahmed, Natarajan, and Rao~\cite{anr1974} was
introduced precisely as a fast approximation to this KLT. Clarke later
showed that the approximation becomes exact in the limit
$\rho\to1$, where the KLT tends to the DCT-II~\cite{clarke1981}, and
that as $\rho\to0$ the KLT instead tends to the
DST-I~\cite{clarke1984}. Jain identified the underlying tridiagonal
structure~\cite{jain1976,jain1979}. Away from these limiting cases,
the exact KLT was long treated as an
unstructured dense matrix: explicit root-finding procedures for its
frequencies exist~\cite{torun2013tsp}, but for five decades no fast,
butterfly-style algorithm for the exact AR(1) KLT appeared in the
literature---in contrast to the DCT-II itself, whose factorization by
Chen, Smith, and Fralick~\cite{chen1977} into a butterfly stage and
half-size DCT-II and DCT-IV cores launched four decades of
fast-transform design~\cite{raoyip1990,britanak2006}.

It was recently shown~\cite{reznik2026tsp} that this contrast is not
intrinsic: the exact AR(1) KLT can be factorized \emph{directly}. For
even orders $N=2M$, the transform reduces to a butterfly stage, two
half-order copies of itself, and orthogonal correction stages
generated by rank-one boundary perturbations of the covariance,
\begin{equation}
\W_N(\rho) \;=\; \PP_N
\Big(\widehat\Q^{(s)}_M\,\W_M(\rho)\;\oplus\;
\widehat\Q^{(a)}_M\,\W_M(\rho)\Big)\,\B_N ,
\label{eq:tsprec}
\end{equation}
a companion identity handles odd orders, and the recursion continues
through every order $N\ge2$, running in $O(N\log N)$ operations when
the Cauchy-structured correction stages are applied by the fast
multipole method~\cite{reznik2026tsp}. The first stages of this
factorization---the sum/difference butterfly $\B_N$, the parity split
it realizes, and the output interleaving $\PP_N$---are \emph{exactly
the same} as the stages that open Chen's factorization of the
DCT-II; indeed, as $\rho\to1$ the recursion degenerates to Chen's
decomposition, with the half-order KLTs turning into the fixed
half-size DCT-II and DCT-IV cores~\cite{reznik2026tsp}.

This structural coincidence poses a complementary problem, which the
present paper solves: to derive a \emph{more direct mapping of the
KLT onto the DCT-II}. In the recursive factorization
\eqref{eq:tsprec}, the trigonometric content is generated anew at
every level, and the corrections
$\widehat\Q^{(s)}_M,\widehat\Q^{(a)}_M$ relate the KLT to its own
half-order copies. One may instead ask for a factorization in which
the trigonometric content is carried entirely by the \emph{fixed}
transforms of Chen's graph, and the whole $\rho$-dependence is
concentrated in corrections attached to them---so that fast
algorithms for the exact KLT follow by reusing, unchanged, the known
fast factorizations of the DCT-II and
DCT-IV~\cite{chen1977,britanak2006}.

The structural fact behind Chen's factorization---and behind both
factorizations of the KLT---is a mirror symmetry: the AR(1)
covariance is unchanged when its rows and columns are both read in
reverse order, so each of its eigenvectors is either
\emph{symmetric}, $w_k=w_{N-1-k}$, or \emph{antisymmetric},
$w_k=-w_{N-1-k}$, about the block center. After the butterfly
$\B_N$, which forms the boundary-paired sums
$u_k=(x_k+x_{N-1-k})/\sqrt2$ and differences
$v_k=(x_k-x_{N-1-k})/\sqrt2$, the KLT therefore splits into two
independent half-size eigenproblems. The step that carries the
mapping onto the DCT-II is taken on the inverse covariance, which is
tridiagonal: we show (Lemma~\ref{lem:folds}) that each folded
half-size block is \emph{exactly} the fixed tridiagonal matrix
diagonalized by the DCT-II, respectively the DCT-IV, scaled and
shifted, plus a rank-one term localized at the block boundary. Rotating each block into the corresponding fixed DCT basis
thus leaves a diagonal-plus-rank-one eigenproblem, and its
diagonalizer is the entire $\rho$-dependent content of the
transform.

The result (Theorem~\ref{thm:folded}) is the exact factorization
\begin{equation}
\W_N(\rho) \;=\; \PP_N
\Big(\Q^{(s)}_{M}\,\C^{\mathrm{II}}_{M}\;\oplus\;
\Q^{(a)}_{M}\,\C^{\mathrm{IV}}_{M}\Big)\,\B_N,
\qquad M=N/2:
\label{eq:mainres}
\end{equation}
the exact KLT is the fast DCT-II flow graph---butterfly plus fixed
DCT-II and DCT-IV halves---completed by two orthogonal corrections
$\Q^{(s)}_M,\Q^{(a)}_M$, the eigenvector matrices of explicit
diagonal-plus-rank-one matrices. All quantities in
\eqref{eq:mainres} beyond the Chen structure are design-time
constants: in the setting most relevant to applications, $\rho$ is
estimated once from the source class, all design computations are
performed offline, and only the resulting flow graph is applied to
each data block. Chen's algorithm is recovered verbatim in the limit
$\rho\to1$ (Corollary~\ref{cor:chen}), and the corrections are made
fully explicit at the short lengths used in practice. At $N=4$ each
correction is a single plane rotation, and the transform merges into
the fast 4-point DCT-II graph with two $\rho$-tunable closed-form
angles (Sec.~\ref{sec:n4}). At $N=8$ every correction entry is
written in radicals---the coupling weights are determined by the
core eigenvalues, the correction eigenvalues by a single quartic
serving both parity branches, and the entries by a one-line
eigenvector formula---and each correction factors into six explicit
Givens rotations (Sec.~\ref{sec:code}). An operation-count analysis
(Tables~\ref{tab:ops4} and~\ref{tab:ops}) prices exact
$\rho$-adaptivity against the fixed DCT-II: a single extra
multiplication at $N=4$, about a factor of two at $N=8$, with the overhead largest
at intermediate block sizes and vanishing again at large $N$, where
the rank-one eigenstructure of the corrections admits
accuracy-controlled $O(N\log N)$ application
(Sec.~\ref{sec:largeN}).

The paper is organized as follows. Sec.~\ref{sec:background} fixes
notation and, to keep the presentation self-sufficient, restates the
phase-equation solution of the AR(1) eigenproblem from
Ref.~[\citenum{reznik2026tsp}], repeating its short proof.
Sec.~\ref{sec:secular} works out the parity split, its half-size
frequency equations, and a no-go observation---no interior value of
$\rho$ reduces the KLT exactly to any fixed uniform-grid
transform---which is what makes the corrections unavoidable.
Sec.~\ref{sec:fact} derives the factorization \eqref{eq:mainres}.
Sec.~\ref{sec:budget} describes the structure of the corrections,
compares the exact realizations and their operation counts, and
works out the cases $N=4$ and $N=8$; Sec.~\ref{sec:largeN} treats
large $N$, and Sec.~\ref{sec:conclusion} concludes.

\section{Preliminaries}
\label{sec:background}

\subsection{Notation}

Following the conventions of Britanak, Yip, and
Rao~\cite{britanak2006}, transform matrices carry their \emph{order}
as a subscript and their \emph{type} as a roman superscript:
$\C^{\mathrm{II}}_M$ and $\C^{\mathrm{IV}}_M$ denote the orthonormal
DCT-II and DCT-IV matrices of order $M$,
\begin{equation}
\big[\C^{\mathrm{II}}_M\big]_{n,k}
= \sqrt{\tfrac{2}{M}}\,\beta_n\cos\tfrac{\pi n(2k+1)}{2M},
\qquad
\big[\C^{\mathrm{IV}}_M\big]_{n,k}
= \sqrt{\tfrac{2}{M}}\,\cos\tfrac{\pi(2n+1)(2k+1)}{4M},
\label{eq:dctdefs}
\end{equation}
with $\beta_0=1/\sqrt2$ and $\beta_n=1$ otherwise, $n,k=0,\dots,M-1$.
We also use the orthonormal DST-I matrix $\mathbf{S}^{\mathrm{I}}_M$,
$\big[\mathbf{S}^{\mathrm{I}}_M\big]_{n,k}
=\sqrt{2/(M{+}1)}\,\sin\tfrac{\pi(n+1)(k+1)}{M+1}$, which appears
below as the $\rho\to0$ limit of the KLT.
$\I_M$ is the identity and $\J_M$ the counter-identity
(order-reversal) matrix of order $M$; following
Ref.~[\citenum{britanak2006}], empty positions in block matrices denote
zero blocks. Parenthesized
superscripts $(s)$ and $(a)$ label objects attached to the
symmetric and antisymmetric halves of the parity decomposition; thus
$\Q^{(s)}_M$ is an order-$M$ matrix acting on the symmetric half.
Finally, $e_0=(1,0,\dots,0)^{T}$ denotes the first unit vector of
length $M$; in the constructions below it marks the \emph{block
boundary} of the half-size problems, and its role is not incidental:
the AR(1) eigenproblem differs from a translation-invariant one only
through its boundary conditions, and correspondingly the entire
$\rho$-dependence of the KLT beyond the fixed cores will enter
through rank-one terms supported on $e_0$.

Both DCTs in \eqref{eq:dctdefs} are eigenvector matrices of simple
tridiagonal matrices~\cite{strang1999}: with
\begin{equation}
\LM^{\mathrm{II}}_M=
\begin{pmatrix}
1 & -1 & & \\ -1 & 2 & \ddots & \\ & \ddots & \ddots & -1\\ & & -1 & 1
\end{pmatrix},
\qquad
\LM^{\mathrm{IV}}_M=
\begin{pmatrix}
1 & -1 & & \\ -1 & 2 & \ddots & \\ & \ddots & \ddots & -1\\ & & -1 & 3
\end{pmatrix},
\label{eq:generators}
\end{equation}
we have
$\C^{\mathrm{II}}_M\LM^{\mathrm{II}}_M(\C^{\mathrm{II}}_M)^{T}
=\Lambda^{\mathrm{II}}_M$ and
$\C^{\mathrm{IV}}_M\LM^{\mathrm{IV}}_M(\C^{\mathrm{IV}}_M)^{T}
=\Lambda^{\mathrm{IV}}_M$, with diagonal eigenvalue matrices
\begin{equation}
\Lambda^{\mathrm{II}}_M
=\mathop{\mathrm{diag}}_{n=0,\dots,M-1}
\Big(2-2\cos\tfrac{n\pi}{M}\Big),
\qquad
\Lambda^{\mathrm{IV}}_M
=\mathop{\mathrm{diag}}_{n=0,\dots,M-1}
\Big(2-2\cos\tfrac{(2n+1)\pi}{2M}\Big).
\label{eq:lambdas}
\end{equation}

\subsection{The AR(1) eigenproblem and its phase equation}
\label{sec:phase}

Let $\R_N(\rho)=[\rho^{|i-j|}]_{i,j=0}^{N-1}$ be the AR(1)
covariance, $\rho\in(0,1)$, and let $\W_N(\rho)$ denote its KLT: the
orthogonal matrix whose rows are the eigenvectors of $\R_N(\rho)$,
ordered by decreasing eigenvalue. The inverse covariance is exactly
tridiagonal~\cite{jain1976},
\begin{equation}
\T_N(\rho) \triangleq (1-\rho^{2})\,\R_N(\rho)^{-1} =
\begin{pmatrix}
1 & -\rho & & \\
-\rho & 1{+}\rho^{2} & \ddots & \\
 & \ddots & \ddots & -\rho\\
 & & -\rho & 1
\end{pmatrix},
\label{eq:tridiag}
\end{equation}
so $\R_N$ and $\T_N$ share their eigenvectors, with eigenvalues
related by $\lambda=(1-\rho^{2})/\mu$ and the order reversed. The
eigenproblem $\T_N w=\mu w$ is a three-term recurrence closed by
its boundary rows. The following proposition,
established in Ref.~[\citenum{reznik2026tsp}], solves it in a form
organized around a single \emph{phase equation}; we repeat its short
proof to keep the presentation self-contained.

\begin{proposition}[{Phase equation, Ref.~[\citenum{reznik2026tsp}]}]
\label{prop:phaseeq}
Define the boundary phase
\begin{equation}
\theta(\omega) \triangleq \arg\!\big(1-\rho e^{-j\omega}\big)
= \arctan\frac{\rho\sin\omega}{1-\rho\cos\omega},
\qquad \omega\in(0,\pi).
\label{eq:theta}
\end{equation}
For each $m=1,\dots,N$ the equation
\begin{equation}
(N+1)\,\omega + 2\,\theta(\omega) = m\,\pi
\label{eq:phaseeq}
\end{equation}
has a unique root $\omega_m\in(0,\pi)$, and
$\omega_1<\dots<\omega_N$. The eigenpairs of $\R_N(\rho)$ are
\begin{equation}
\lambda_m=\frac{1-\rho^{2}}{1-2\rho\cos\omega_m+\rho^{2}},
\qquad
w^{(m)}_k = c_m\,\sin\!\big(\omega_m(k+1)+\theta(\omega_m)\big),
\quad k=0,\dots,N-1,
\label{eq:eigpairs}
\end{equation}
with $\lambda_1>\dots>\lambda_N$ and unit-normalization constants
$c_m>0$ given explicitly by \eqref{eq:cm} below.
Moreover, $w^{(m)}$ is symmetric ($w_k=w_{N-1-k}$) for odd $m$ and
antisymmetric ($w_k=-w_{N-1-k}$) for even $m$.
\end{proposition}

\begin{proof}
The interior rows of $\T_N w=\mu w$ read
$-\rho w_{k-1}+(1+\rho^{2})w_k-\rho w_{k+1}=\mu w_k$ and are
satisfied by $w_k=\sin(\omega k+\varphi)$ for any phase $\varphi$,
provided $\mu=1+\rho^{2}-2\rho\cos\omega$; this gives the eigenvalue
formula in \eqref{eq:eigpairs} via $\lambda=(1-\rho^{2})/\mu$. The
first and last rows of \eqref{eq:tridiag} are equivalent to extending
the interior recurrence with fictitious samples obeying
\begin{equation}
w_{-1}=\rho\,w_0,\qquad w_{N}=\rho\,w_{N-1}
\label{eq:bc}
\end{equation}
(e.g., the row-0 equation
$w_0-\rho w_1=\mu w_0$ coincides with the interior equation at $k=0$
once $w_{-1}=\rho w_0$ is substituted). The left condition
determines the phase: $\sin(\varphi-\omega)=\rho\sin\varphi$ gives
$\tan\varphi=\sin\omega/(\cos\omega-\rho)$, whose solution in
$(0,\pi)$ is $\varphi=\arg(e^{j\omega}-\rho)
=\arg\big(e^{j\omega}(1-\rho e^{-j\omega})\big)=\omega+\theta(\omega)$,
so $w_k=\sin(\omega(k{+}1)+\theta)$. Substituting into the right
condition and writing $\alpha=\omega(N{+}1)+\theta$,
$\sin\alpha=\rho\sin(\alpha-\omega)$ rearranges to
$\tan\alpha=-\rho\sin\omega/(1-\rho\cos\omega)=-\tan\theta$, i.e.\
$\alpha=m\pi-\theta$, which is \eqref{eq:phaseeq}. For uniqueness and
ordering, differentiate \eqref{eq:theta}:
$\theta'(\omega)=\rho(\cos\omega-\rho)/(1-2\rho\cos\omega+\rho^{2})
\ge-\rho/(1{+}\rho)>-\tfrac12$, so the left side of
\eqref{eq:phaseeq} is strictly increasing with slope at least $N$,
running from $0$ at $\omega=0$ to $(N{+}1)\pi$ at $\omega=\pi$; it
crosses each level $m\pi$, $m=1,\dots,N$, exactly once. Since
$\lambda(\omega)$ is strictly decreasing on $(0,\pi)$, the $N$
eigenvalues are distinct and the system of eigenvectors is complete.
The normalization is likewise explicit: by the finite sum
$\sum_{k=0}^{N-1}\sin^{2}(\omega(k{+}1)+\theta)
=\tfrac{N}{2}-\tfrac{\sin N\omega}{2\sin\omega}
\cos\!\big((N{+}1)\omega+2\theta\big)$
and \eqref{eq:phaseeq}, which sets the cosine to $\cos m\pi=(-1)^m$
at each root,
\begin{equation}
c_m^{-2} \;=\; \frac{N}{2}
\;-\;(-1)^{m}\,\frac{\sin N\omega_m}{2\sin\omega_m}\,,
\label{eq:cm}
\end{equation}
which at $\rho=0$ (where $\omega_m=\tfrac{m\pi}{N+1}$, so
$\sin N\omega_m=(-1)^{m+1}\sin\omega_m$) recovers the DST-I value
$c_m=\sqrt{2/(N{+}1)}$.
Finally, using \eqref{eq:phaseeq},
\[
w^{(m)}_{N-1-k}=c_m\sin\big(\omega_m(N-k)+\theta\big)
=c_m\sin\big(m\pi-\omega_m(k{+}1)-\theta\big)
=(-1)^{m+1}\,w^{(m)}_k ,
\]
which proves the parity claim.
\end{proof}

This has a simple interpretation: each eigenmode is a sinusoid
reflecting between the block ends, and each boundary condition in
\eqref{eq:bc} contributes the reflection phase $\theta(\omega)$ to
the total phase budget $m\pi$ in \eqref{eq:phaseeq}. As $\rho\to0$,
where $\theta\to0$, the modes tend to the DST-I basis
$\mathbf{S}^{\mathrm{I}}_N$. As $\rho\to1$, where
$\theta\to(\pi-\omega)/2$, they become the DCT-II
basis~\cite{clarke1981,clarke1984}, in agreement with the classical
limits.

\section{Parity splitting}
\label{sec:secular}

Let $N$ be even, $M=N/2$. By Proposition~\ref{prop:phaseeq}, the
eigenvectors of $\R_N(\rho)$ split into symmetric and antisymmetric
classes; this parity decomposition is realized by the orthonormal
butterfly $\B_N$
mapping $x\mapsto(u_0,\dots,u_{M-1},v_0,\dots,v_{M-1})$, with
$u_k=(x_k+x_{N-1-k})/\sqrt2$ and $v_k=(x_k-x_{N-1-k})/\sqrt2$:
\begin{equation}
\B_N = \frac{1}{\sqrt2}
\begin{pmatrix}\I_M & \J_M\\ \I_M & -\J_M\end{pmatrix}
= \begin{pmatrix}\I_M & \\ & \J_M\end{pmatrix}
\widetilde\B_N,
\qquad
\widetilde\B_N = \frac{1}{\sqrt2}
\begin{pmatrix}\I_M & \J_M\\ \J_M & -\I_M\end{pmatrix}.
\label{eq:butterfly}
\end{equation}
The second form is the implementation view: $\widetilde\B_N$ is the
fully symmetric butterfly, whose difference outputs appear in
mirrored order and are restored to natural order by the reversal
$\J_M$; this is the ``reorder'' stage in the flow graphs below.
Up to the scaling $1/\sqrt2$, $\widetilde\B_N$ is exactly the
butterfly matrix of Chen, Smith, and Fralick~\cite{chen1977}, who
write it with entries $\pm1$ and absorb the reversal into their
bit-reversed output ordering.
Because symmetric eigenvectors are constant on the pairs
$(k,N{-}1{-}k)$ and antisymmetric ones change sign across them,
$\B_N$ block-diagonalizes $\T_N(\rho)$, with the two blocks carrying
the modes of odd and even $m$, respectively. The phase equation itself
then splits by the parity of $m$ into two half-size,
arctangent-free \emph{frequency equations}---the discrete counterparts
of the Ray--Driver pair~\cite{reznik2026tsp}.

\begin{proposition}[{Frequency equations, Ref.~[\citenum{reznik2026tsp}]}]
\label{prop:secular}
For $\rho\in(0,1)$, the frequencies of the symmetric modes (odd $m$)
are exactly the roots in $(0,\pi)$ of
\begin{equation}
\cos\tfrac{(N+1)\omega}{2}=\rho\,\cos\tfrac{(N-1)\omega}{2},
\label{eq:secular-s}
\end{equation}
and those of the antisymmetric modes (even $m$) the roots of
\begin{equation}
\sin\tfrac{(N+1)\omega}{2}=\rho\,\sin\tfrac{(N-1)\omega}{2}.
\label{eq:secular-a}
\end{equation}
\end{proposition}

\begin{proof}
Equation \eqref{eq:phaseeq} reads
$\tfrac{(N+1)\omega}{2}+\theta=\tfrac{m\pi}{2}$; for odd $m$ the
cosine of the left side vanishes. Multiplying by
$|1-\rho e^{-j\omega}|$ and using
$e^{j\theta}|1-\rho e^{-j\omega}|=1-\rho e^{-j\omega}$,
\[
\mathrm{Re}\!\left[e^{\,j\frac{(N+1)\omega}{2}}
\big(1-\rho e^{-j\omega}\big)\right]
= \cos\tfrac{(N+1)\omega}{2}-\rho\cos\tfrac{(N-1)\omega}{2}=0 .
\]
For even $m$ the imaginary part of the same quantity vanishes,
giving \eqref{eq:secular-a}.
\end{proof}

Written in terms of $z=e^{j\omega}$,
\eqref{eq:secular-s} and \eqref{eq:secular-a} are polynomial
equations; for even $N$, each reduces to a polynomial of degree $M$
in $\cos\omega$. The frequencies---and with them every factorization
constant constructed below---are therefore algebraic functions of
$\rho$, and for $N\le8$, where these polynomials have degree at most
four, they are expressible in radicals (Secs.~\ref{sec:n4} and~\ref{sec:code}).
More important here is what these equations rule out. For
every $\rho\in(0,1)$, no eigenfrequency lies on the DST-I grid
$\{m\pi/(N{+}1)\}$, since this would force $\theta(\omega_m)=0$ in
\eqref{eq:phaseeq}, which is impossible for $\rho>0$; nor does any
eigenfrequency lie on the DCT-II grid $\{n\pi/N\}$: substituting
$\omega=2l\pi/N$ into \eqref{eq:secular-s} gives
$(-1)^{l}\cos\tfrac{l\pi}{N}=\rho(-1)^{l}\cos\tfrac{l\pi}{N}$,
forcing $\rho=1$, and the remaining parity combinations likewise
force $\rho\in\{-1,1\}$. Thus no interior value of $\rho$ reduces the
KLT exactly to a fixed uniform-grid trigonometric transform. The
appropriate fast structure is instead an exact factorization with
fixed fast cores and $\rho$-dependent corrections, which we now
derive.

\section{The exact factorization of the KLT}
\label{sec:fact}

The derivation takes three elementary steps: split the tridiagonal
$\T_N(\rho)$ into two half-size blocks by symmetry; recognize each
block, entry by entry, as one of the fixed matrices
\eqref{eq:generators} scaled and shifted, plus a rank-one term
(Lemma~\ref{lem:folds}); and rotate each block by its DCT, which
leaves only a diagonal-plus-rank-one matrix to diagonalize
(Theorem~\ref{thm:folded}). The fold itself is common to
this paper and to the recursive factorization of
Ref.~[\citenum{reznik2026tsp}]; the two constructions part ways in how
the folded blocks are expressed.

\begin{lemma}[Parity split of the inverse covariance]
\label{lem:folds}
Let $N$ be even, $M=N/2$. Then
\begin{equation}
\B_N\,\T_N(\rho)\,\B_N^{T}
=\begin{pmatrix}\T^{(s)}_M & \\[2pt] & \T^{(a)}_M\end{pmatrix},
\label{eq:blockdiag}
\end{equation}
where the two half-size blocks satisfy the exact identities
\begin{equation}
\T^{(s)}_M = \rho\,\LM^{\mathrm{II}}_M+(1-\rho)^{2}\I_M
+\rho(1-\rho)\,e_0e_0^{T},
\qquad
\T^{(a)}_M = \rho\,\LM^{\mathrm{IV}}_M+(1-\rho)^{2}\I_M
+\rho(1-\rho)\,e_0e_0^{T}.
\label{eq:folds}
\end{equation}
\end{lemma}

\begin{proof}
Write $T_{ij}$ for the entries of $\T_N(\rho)$, and note from
\eqref{eq:tridiag} that $T_{N-1-i,\,N-1-j}=T_{ij}$. A direct
computation with the rows of $\B_N$ in \eqref{eq:butterfly} then
gives, for $0\le i,j\le M{-}1$,
\[
\big(\B_N\T_N\B_N^{T}\big)_{i,j}=T_{ij}+T_{i,N-1-j},
\qquad
\big(\B_N\T_N\B_N^{T}\big)_{M+i,\,M+j}=T_{ij}-T_{i,N-1-j},
\]
and the two off-diagonal blocks vanish, proving
\eqref{eq:blockdiag} with
$\big(\T^{(s)}_M\big)_{ij}=T_{ij}+T_{i,N-1-j}$ and
$\big(\T^{(a)}_M\big)_{ij}=T_{ij}-T_{i,N-1-j}$. Since $\T_N$ is
tridiagonal, the cross term $T_{i,N-1-j}$ vanishes for all
$i,j\le M{-}1$ except at $i=j=M{-}1$, where it equals
$T_{M-1,M}=-\rho$. Hence both blocks equal the leading $M\times M$
block of $\T_N$---diagonal $(1,1{+}\rho^{2},\dots,1{+}\rho^{2})$,
off-diagonals $-\rho$---with the last diagonal entry shifted to
$1+\rho^{2}\mp\rho$. It remains to compare with the right sides of
\eqref{eq:folds} entrywise. Interior diagonals:
$2\rho+(1-\rho)^{2}=1+\rho^{2}$; off-diagonals: $-\rho$; both match.
Last diagonal entries: the corner values $1$ of
$\LM^{\mathrm{II}}_M$ and $3$ of $\LM^{\mathrm{IV}}_M$ in
\eqref{eq:generators} give $\rho\cdot1+(1-\rho)^{2}=1+\rho^{2}-\rho$
and $\rho\cdot3+(1-\rho)^{2}=1+\rho^{2}+\rho$; both match. First
diagonal entries: $\rho\cdot1+(1-\rho)^{2}=1-\rho+\rho^{2}$ differs
from the required $1$ by exactly $\rho(1-\rho)$---which is supplied
by the rank-one term $\rho(1-\rho)e_0e_0^{T}$. All entries agree,
proving \eqref{eq:folds}.
\end{proof}

\begin{remark}[Relation to the recursive fold]
\label{rem:tspfold}
The fold \eqref{eq:blockdiag} is also the starting point of the
recursive factorization \eqref{eq:tsprec}~\cite{reznik2026tsp}.
There the same blocks are written as
$\T^{(s)}_M=\T_M(\rho)-\rho(1-\rho)\,e_{M-1}e_{M-1}^{T}$ and
$\T^{(a)}_M=\T_M(\rho)+\rho(1+\rho)\,e_{M-1}e_{M-1}^{T}$---the
half-size AR(1) matrix $\T_M(\rho)$ perturbed at the last coordinate
$e_{M-1}$, nearest the fold---which drives the recursion through the
half-order KLT $\W_M(\rho)$. The identities \eqref{eq:folds} instead
express the same blocks through the \emph{fixed} matrices
$\LM^{\mathrm{II}}_M,\LM^{\mathrm{IV}}_M$, with the rank-one term
moved to the outer boundary $e_0$: this is the splitting that maps
the KLT directly onto the fixed DCT bases. Both decompositions are
exact; they simply localize the $\rho$-dependence at opposite ends
of the folded block.
\end{remark}

Lemma~\ref{lem:folds} exhibits each half-size block as a fixed
matrix from \eqref{eq:generators}, scaled and shifted, plus a
rank-one term at the boundary coordinate $e_0$. Rotating each block
into its core's eigenbasis therefore leaves a diagonal matrix plus a
rank-one term. Define
\begin{equation}
\begin{aligned}
\G^{(s)}_M &\,\triangleq\,
\C^{\mathrm{II}}_M\,\T^{(s)}_M\big(\C^{\mathrm{II}}_M\big)^{T}
= \D^{(s)}_M+\rho(1-\rho)\,q^{(s)}\big(q^{(s)}\big)^{T},\\
\G^{(a)}_M &\,\triangleq\,
\C^{\mathrm{IV}}_M\,\T^{(a)}_M\big(\C^{\mathrm{IV}}_M\big)^{T}
= \D^{(a)}_M+\rho(1-\rho)\,q^{(a)}\big(q^{(a)}\big)^{T},
\end{aligned}
\label{eq:Gdefs}
\end{equation}
where, using \eqref{eq:lambdas},
\begin{equation}
\D^{(s)}_M=\rho\,\Lambda^{\mathrm{II}}_M+(1-\rho)^{2}\I_M,
\qquad
\D^{(a)}_M=\rho\,\Lambda^{\mathrm{IV}}_M+(1-\rho)^{2}\I_M,
\qquad
q^{(s)}=\C^{\mathrm{II}}_Me_0,
\qquad
q^{(a)}=\C^{\mathrm{IV}}_Me_0.
\label{eq:Ddefs}
\end{equation}
The correction factors are now defined as the diagonalizers of
\eqref{eq:Gdefs}: for $0<\rho<1$ and $x\in\{s,a\}$, $\Q^{(x)}_M$
is the orthogonal
matrix whose rows are the unit eigenvectors of $\G^{(x)}_M$, ordered
by increasing eigenvalue, so that
\begin{equation}
\Q^{(x)}_M\,\G^{(x)}_M\big(\Q^{(x)}_M\big)^{T}=\Delta^{(x)}_M,
\qquad \Delta^{(x)}_M \ \text{diagonal, ascending}.
\label{eq:Qdef}
\end{equation}

\begin{theorem}[Factorization of the KLT]
\label{thm:folded}
Let $0<\rho<1$, let $N$ be even, $M=N/2$, and let
$\Q^{(s)}_M,\Q^{(a)}_M$ be as in \eqref{eq:Qdef}. Then the AR(1) KLT factors exactly as
\begin{equation}
\W_N(\rho) \;=\; \PP_N
\begin{pmatrix}
\Q^{(s)}_M\,\C^{\mathrm{II}}_M & \\[2pt]
 & \Q^{(a)}_M\,\C^{\mathrm{IV}}_M
\end{pmatrix}
\B_N ,
\label{eq:factor}
\end{equation}
where $\PP_N$ is the interleaving permutation matrix,
$[\PP_N]_{2i,\,i}=[\PP_N]_{2i+1,\,M+i}=1$ for $i=0,\dots,M{-}1$ and
zero otherwise; that is, the KLT outputs of even index
$y_0,y_2,\dots$ are the outputs of the symmetric branch in order, and
those of odd index $y_1,y_3,\dots$ are the outputs of the
antisymmetric branch in order.
\end{theorem}

\begin{proof}
Call $\mathbf{F}$ the matrix on the right of \eqref{eq:factor} without
$\PP_N$. Combining \eqref{eq:blockdiag} with
\eqref{eq:Gdefs}--\eqref{eq:Qdef} gives
$\mathbf{F}\,\T_N(\rho)\,\mathbf{F}^{T}
=\mathrm{diag}\big(\Delta^{(s)}_M,\Delta^{(a)}_M\big)$, a diagonal
matrix; the rows of $\mathbf{F}$ are therefore a complete
orthonormal system of eigenvectors of $\T_N$, and hence of $\R_N$,
with the first $M$ rows spanning the symmetric modes and the last $M$
rows the antisymmetric ones, each half internally ordered by
increasing eigenvalue of $\T_N$, i.e.\ by \emph{decreasing}
eigenvalue of $\R_N$. It remains to show that the KLT ordering of all
$N$ modes alternates between the halves, starting with the symmetric
one. By Proposition~\ref{prop:phaseeq}, the symmetric modes are the
phase-equation roots of odd $m$ and the antisymmetric ones those of
even $m$; the roots $\omega_m$ strictly increase with $m$ and
$\lambda(\omega)$ strictly decreases, so
$\lambda_1>\lambda_2>\dots>\lambda_N$ alternates
symmetric--antisymmetric. The permutation that realizes this
interleaving of the two internally ordered halves is exactly
$\PP_N$.
\end{proof}

\begin{corollary}[Chen's factorization as the $\rho\to1$ limit]
\label{cor:chen}
In the limit $\rho\to1$, the rank-one terms in \eqref{eq:folds}
and \eqref{eq:Gdefs} vanish. The branch diagonals become
$\D^{(s)}_M=\Lambda^{\mathrm{II}}_M$ and
$\D^{(a)}_M=\Lambda^{\mathrm{IV}}_M$, already diagonal with ascending
entries, so
$\Q^{(s)}_M=\Q^{(a)}_M=\I_M$ and \eqref{eq:factor} degenerates to
\begin{equation}
\W_N=\PP_N
\begin{pmatrix}\C^{\mathrm{II}}_M & \\[2pt]
 & \C^{\mathrm{IV}}_M\end{pmatrix}\B_N
\label{eq:chen}
\end{equation}
---precisely the classical Chen--Smith--Fralick factorization of the
DCT-II~\cite{chen1977}. Chen's algorithm is thus the degenerate
endpoint of an exact $\rho$-parametric factorization of the AR(1)
KLT.
\end{corollary}

Theorem~\ref{thm:folded} is the direct counterpart of the
recursive factorization \eqref{eq:tsprec}: both open with the
same butterfly $\B_N$ and interleaving $\PP_N$, but there the parity
halves are carried by the half-order KLTs $\W_M(\rho)$, and here by
the fixed cores $\C^{\mathrm{II}}_M,\C^{\mathrm{IV}}_M$, with all
$\rho$-dependence pushed into the terminal corrections
$\Q^{(s)}_M,\Q^{(a)}_M$. Since both middle factors diagonalize the
same blocks \eqref{eq:blockdiag} with the same eigenvalue ordering,
the two agree termwise up to row signs:
$\Q^{(s)}_M\C^{\mathrm{II}}_M=\widehat\Q^{(s)}_M\W_M(\rho)$ and
$\Q^{(a)}_M\C^{\mathrm{IV}}_M=\widehat\Q^{(a)}_M\W_M(\rho)$. The
present corrections thus absorb, in a single stage, the entire
recursion tail of \eqref{eq:tsprec}. This is what enables the reuse
of existing fast transforms: the cores in \eqref{eq:factor} may be
computed by \emph{any} known fast factorization of the DCT-II and
DCT-IV~\cite{chen1977,britanak2006}, whereas the recursion of
Ref.~[\citenum{reznik2026tsp}] generates its own trigonometric stages.

\section{Implementing the correction factors}
\label{sec:budget}

Everything specific to the KLT beyond Chen's fixed structure is
concentrated in the two correction factors $\Q^{(s)}_M,\Q^{(a)}_M$
defined by \eqref{eq:Gdefs}--\eqref{eq:Qdef}. This section describes
their structure and exact application
(Sec.~\ref{sec:exact-impl}), compares the available realizations
and their operation counts (Sec.~\ref{sec:opcounts}), and works the
construction out in closed form at the two short lengths $N=4$
(Sec.~\ref{sec:n4}) and $N=8$ (Sec.~\ref{sec:code}); what to do
when $N$ is large is taken up in Sec.~\ref{sec:largeN}.

\subsection{Structure of the corrections and exact application}
\label{sec:exact-impl}

For $0<\rho<1$, the matrices \eqref{eq:Gdefs} are of
diagonal-plus-rank-one form, a classical
eigenproblem~\cite{bunch1978}. Fix a branch $x\in\{s,a\}$ and
abbreviate $\G=\G^{(x)}_M=\D+\sigma qq^{T}$, where
$\D=\D^{(x)}_M=\mathrm{diag}(d_0,\dots,d_{M-1})$, with
$d_0<\dots<d_{M-1}$ by \eqref{eq:Ddefs}, $q=q^{(x)}$, and
$\sigma=\rho(1-\rho)>0$. The eigenvalues $\mu_0<\dots<\mu_{M-1}$ of
$\G$ are the roots of the scalar function
\begin{equation}
f(\mu)=1+\sigma\sum_{i=0}^{M-1}\frac{q_i^{2}}{d_i-\mu},
\label{eq:secfun}
\end{equation}
strictly interlacing the poles: $d_i<\mu_i<d_{i+1}$ for
$i=0,\dots,M-2$, with
$d_{M-1}<\mu_{M-1}<d_{M-1}+\sigma\|q\|^{2}$. The corresponding unit
eigenvectors---the rows of $\Q^{(x)}_M$ in \eqref{eq:Qdef}---are~\cite{bunch1978}
\begin{equation}
\big[\Q^{(x)}_M\big]_{i,:}\;\propto\;
\big(\D-\mu_i\I_M\big)^{-1}q .
\label{eq:eigvecs}
\end{equation}
Thus the corrections are \emph{explicit} once the $M$ scalar roots of
\eqref{eq:secfun} are found---an offline, one-time computation for
the known $\rho$ (for $M\le4$, the roots are expressible in radicals;
cf.\ Sec.~\ref{sec:secular}). The endpoint cases $\rho\to0$ and $\rho\to1$ follow from the
limits already described. Two exact runtime options are immediate.
First, $\Q^{(x)}_M$ may simply be applied as a dense
$M\times M$ multiply, at $M^{2}$ multiplications per correction on
top of the fast cores. Second, each $\Q^{(x)}_M$, being an $M\times M$
rotation with $M(M{-}1)/2$ degrees of freedom, factors exactly into
$M(M{-}1)/2$ Givens rotations (six at $M=4$) by QR elimination, with
all angles computed offline; in lifting form this is preferable when
multiplierless or integer designs are the goal. Either way the
algorithm is the exact KLT, as guaranteed by
Theorem~\ref{thm:folded}. These and two further exact realizations
are compared next.

\subsection{Realizations and operation counts}
\label{sec:opcounts}

The factorization \eqref{eq:factor} doubles as a cost decomposition.
The butterfly stage is $N$ additions, with its scaling $1/\sqrt2$
absorbable into the factors downstream; the cores carry fixed DCT
arithmetic, to which the fast-DCT literature applies
unchanged~\cite{britanak2006}; and the corrections carry everything
that depends on $\rho$. Because the corrections reduce to identities
at $\rho=1$ (Corollary~\ref{cor:chen}), they are, operation for
operation, the exact price of $\rho$-adaptivity over Chen's fixed
DCT-II: the question ``how much costlier is the exact KLT than the
DCT?'' reduces to the cost of applying $\Q^{(s)}_M,\Q^{(a)}_M$.

Two further exact realizations sharpen the answer. First, since
$\tfrac{1}{\sqrt2}\,\Q^{(s)}_M\C^{\mathrm{II}}_M$ (and its
antisymmetric counterpart) is itself just an $M\times M$ matrix, the
cores, corrections, and butterfly scaling may be \emph{merged} into
a single dense multiply per branch: one addition-only butterfly
followed by two dense half-size blocks at $(N/2)^{2}$
multiplications each, i.e.\ $N^{2}/2$ multiplications and
$N^{2}/2$ additions per complete transform. Merging pays because the butterfly
is the only $\rho$-independent, multiplication-free factor
available---it realizes the reflection symmetry of the
covariance---and it can be exploited only once: splitting a branch
again would require a new dense correction costing as much as the
split saves. The point is quantitative: at $N=16$, one fold with
merged branches costs $128$ multiplications while a recursive step
with dense corrections costs $192$, and full recursion approaches
the dense count ($16\,064$ against $16\,384$ at $N=128$). Second, in the manner of the Arai--Agui--Nakajima
scaled DCT~\cite{arai1988}, each merged row may be normalized by its
largest entry, making one coefficient per row equal to $\pm1$ and
leaving all others of magnitude below one; the resulting positive
diagonal folds into the quantizer and dequantizer of a coding chain
at no cost. This \emph{scaled} form needs three multiplications per
row---$N^{2}/2-N$ in total---and, at $N\le8$, all its coefficients
remain radicals in $\rho$; the $\pm1$ positions are part of the
offline design.
Tables~\ref{tab:ops4} and~\ref{tab:ops} below quantify these
realizations at $N=4$ and $N=8$.

\subsection{The simplest case: $N=4$}
\label{sec:n4}

At $N=2$ the KLT is the $\rho$-independent butterfly. The first
nontrivial length, $N=4$, is the degenerate case of the correction
machinery: for $M=2$ the matrices \eqref{eq:Gdefs} are symmetric
$2\times2$, so each correction $\Q^{(x)}_2$ is a single plane
rotation and no root-finding of any kind arises. Simplest of all is
the merged realization of Sec.~\ref{sec:opcounts}: the products
$\Q^{(s)}_2\C^{\mathrm{II}}_2$ and
$\Q^{(a)}_2\C^{\mathrm{IV}}_2$ are the eigenvector matrices of
the $2\times2$ blocks \eqref{eq:folds} directly,
\begin{equation}
\T^{(s)}_2=\begin{pmatrix}1 & -\rho\\ -\rho & 1+\rho^{2}-\rho\end{pmatrix},
\qquad
\T^{(a)}_2=\begin{pmatrix}1 & -\rho\\ -\rho & 1+\rho^{2}+\rho\end{pmatrix},
\label{eq:folds4}
\end{equation}
and are single plane rotations $G(\psi_s)$, $G(\psi_a)$, where
\begin{equation}
G(\psi)=\begin{pmatrix}\cos\psi & -\sin\psi\\
\sin\psi & \cos\psi\end{pmatrix}.
\label{eq:rot}
\end{equation}
Applying
to \eqref{eq:folds4} the $2\times2$ eigenrotation formula
$\tan2\psi=2b/(a-c)$ for
$\big(\begin{smallmatrix}a&b\\ b&c\end{smallmatrix}\big)$
gives directly
\begin{equation}
\tan2\psi_s=\frac{-2\rho}{1-(1+\rho^{2}-\rho)}=\frac{-2}{1-\rho},
\qquad
\tan2\psi_a=\frac{-2\rho}{1-(1+\rho^{2}+\rho)}=\frac{2}{1+\rho},
\label{eq:angles}
\end{equation}
i.e.\ $\psi_s=\pi/2-\tfrac12\arctan\tfrac{2}{1-\rho}$ and
$\psi_a=\tfrac12\arctan\tfrac{2}{1+\rho}$; the cosine--sine
pairs follow by half-angle identities---rational functions and
square roots in $\rho$.

\begin{figure}[t]
\centering
\begin{tikzpicture}[x=1cm,y=1.1cm,font=\small,
  port/.style={circle,fill,inner sep=1.1pt}]
% inputs
\foreach \i in {0,...,3}{
  \node[left] (x\i) at (0,{3-\i}) {$x_{\i}$};
  \node[port] (px\i) at (0.08,{3-\i}) {};
}
% butterfly outputs
\foreach \r/\lab in {0/u_0,1/u_1,2/v_1,3/v_0}{
  \node[port] (b\r) at (2.7,{3-\r}) {};
  \node[above=0.5pt] at (2.7,{3-\r}) {\scriptsize $\lab$};
}
% butterfly connections ($-$ marks a negated branch)
\draw (px0)--(b0); \draw (px3)--(b0);
\draw (px1)--(b1); \draw (px2)--(b1);
\draw (px1)--(b2);
\draw (px2)--(b2) node[pos=0.8,below=0.5pt,font=\scriptsize]{$-$};
\draw (px0)--(b3);
\draw (px3)--(b3) node[pos=0.8,below=0.5pt,font=\scriptsize]{$-$};
% reorder J2
\node[port] (c2) at (4.3,1) {}; \node[above=0.5pt] at (4.3,1) {\scriptsize $v_0$};
\node[port] (c3) at (4.3,0) {}; \node[above=0.5pt] at (4.3,0) {\scriptsize $v_1$};
\draw (b2)--(c3); \draw (b3)--(c2);
\draw (b0)--(4.9,3); \draw (b1)--(4.9,2);
\draw (c2)--(4.9,1); \draw (c3)--(4.9,0);
% fixed cores
\draw[fill=blue!6] (4.9,1.55) rectangle (6.5,3.45);
\node[align=center] at (5.7,2.5)
  {$\C^{\mathrm{II}}_2$\\[-1pt]{\scriptsize fixed}};
\draw[fill=blue!6] (4.9,-0.45) rectangle (6.5,1.45);
\node[align=center] at (5.7,0.5)
  {$\C^{\mathrm{IV}}_2$\\[-1pt]{\scriptsize fixed}};
% corrections
\foreach \yy in {0,...,3}{ \draw (6.5,\yy)--(6.9,\yy); }
\draw[fill=orange!12] (6.9,1.55) rectangle (8.7,3.45);
\node at (7.8,2.5) {$\Q^{(s)}_2(\rho)$};
\draw[fill=orange!12] (6.9,-0.45) rectangle (8.7,1.45);
\node at (7.8,0.5) {$\Q^{(a)}_2(\rho)$};
% outputs
\foreach \yy/\lab in {3/{y_0\ (n=0)},2/{y_2\ (n=2)},1/{y_1\ (n=1)},0/{y_3\ (n=3)}}{
  \node[port] (o\yy) at (9.5,\yy) {};
  \node[right] at (9.5,\yy) {$\lab$};
  \draw (8.7,\yy)--(o\yy);
}
% stage labels
\node at (1.35,3.95) {\itshape symmetric butterfly};
\node at (3.55,3.95) {\itshape reorder $\J_2$};
\node at (5.7,3.95) {\itshape fixed cores};
\node at (7.8,3.95) {\itshape corrections};
% footnote
\node at (4.75,-1.25)
  {\footnotesize $\Q^{(s)}_2=\Q^{(a)}_2=\I_2$ at $\rho=1$\quad
   (fast 4-point DCT-II)};
\end{tikzpicture}
\caption{Signal-flow graph of the exact 4-point AR(1) KLT
\eqref{eq:factor}: fully symmetric butterfly, reorder stage $\J_2$
restoring the difference coordinates to natural order, fixed cores
$\C^{\mathrm{II}}_2,\C^{\mathrm{IV}}_2$, and the $\rho$-dependent
corrections $\Q^{(s)}_2,\Q^{(a)}_2$. A $-$ sign marks a negated
branch, in the drawing convention of Ref.~[\citenum{chen1977}];
the butterfly factor $1/\sqrt2$ is absorbed downstream in
implementation. In the limit $\rho\to1$ the corrections equal
$\I_2$ and the graph is the fast 4-point DCT-II. In implementation,
the core and correction of each branch merge into a single plane
rotation with the closed-form angle \eqref{eq:angles}.}
\label{fig:flow4}
\end{figure}

The structural form of the 4-point AR(1) KLT is the flow graph of
Fig.~\ref{fig:flow4}: one butterfly stage, the trivial reorder
$\J_2$, the fixed $2\times2$ cores, and the two corrections. In
implementation, the core and correction of each branch merge into
the single rotation $G(\psi_s)$, resp.\ $G(\psi_a)$, of
\eqref{eq:angles}, so the complete transform is the fast 4-point
DCT-II graph with two $\rho$-tunable angles. Implementing
each rotation by three-step lifting (3 multiplications, 3
additions) gives 6 multiplications and 10 additions for every
$\rho$. In the limit $\rho\to1$,
\eqref{eq:angles} gives $(45^{\circ},22.5^{\circ})$, the Chen
angles; the symmetric rotation then degenerates to a scaled
butterfly, recovering the usual DCT-II operation counts. As
$\rho\to0$, they tend to $(\arctan\phi,\tfrac12\arctan 2)$, with
$\phi=(1+\sqrt5)/2$, giving the DST-I. The two formulas exchange
under $\rho\mapsto-\rho$: the substitution $x_k\mapsto(-1)^kx_k$
turns an AR($\rho$) source into an AR($-\rho$) one and swaps the
symmetric and antisymmetric classes.
Table~\ref{tab:ops4} collects the operation counts. Note that at
this size merging beats even the generic count $N^{2}/2=8$: each
merged block is a scaled rotation, so the three-multiplication form
applies. In the same three-multiplication convention, the fixed
fast DCT-II costs five multiplications (its difference-branch
rotation, plus two output scalings); the classical dense-rotation
count is six~\cite{britanak2006}. Exact $\rho$-adaptivity at $N=4$
therefore costs a single multiplication---six against five---and
one extra addition: the sum branch acquires a rotation where the
DCT-II has only scalings.

\begin{table}[t]
\vspace{10pt}
\caption{Operation counts of exact realizations of the $4$-point
AR(1) KLT; all design constants are precomputed offline from $\rho$.
Rotations are counted in three-multiplication form.}
\label{tab:ops4}
\centering
\begin{tabular}{lcc}
\toprule
Realization & mult. & add. \\
\midrule
Dense $\W_4$ & 16 & 12 \\
Fast DCT-II graph $+$ rotation corrections & 12 & 14 \\
Merged branches: two rotations \eqref{eq:angles} & 6 & 10 \\
Scaled merged (diagonal absorbed by quantizer) & 4 & 8 \\
\midrule
DCT-II, fast 4-point (three-mult.\ rotation) & 5 & 9 \\
DCT-II, fast 4-point, classical count~\cite{britanak2006} & 6 & 8 \\
\bottomrule
\end{tabular}
\end{table}

\subsection{The case $N=8$: design in radicals}
\label{sec:code}

For $N=8$ the factorization \eqref{eq:factor} is the flow graph of
Fig.~\ref{fig:flow8}. We emphasize that the construction is exact:
\eqref{eq:factor} is an identity, so the
algorithm of Fig.~\ref{fig:flow8} with the full corrections
reproduces the KLT output exactly (up to arithmetic precision), and
its coding-gain loss relative to the KLT is \emph{identically zero}.

\begin{figure}[t]
\centering
\begin{tikzpicture}[x=1cm,y=0.78cm,font=\small,
  port/.style={circle,fill,inner sep=1.1pt}]
% inputs
\foreach \i in {0,...,7}{
  \node[left] (x\i) at (0,{7-\i}) {$x_{\i}$};
  \node[port] (px\i) at (0.08,{7-\i}) {};
}
% butterfly outputs: rows 0-3 sums u_0..u_3, rows 4-7 differences v_3..v_0
\foreach \r/\lab in {0/u_0,1/u_1,2/u_2,3/u_3,4/v_3,5/v_2,6/v_1,7/v_0}{
  \node[port] (b\r) at (3.1,{7-\r}) {};
  \node[above=0.5pt] at (3.1,{7-\r}) {\scriptsize $\lab$};
}
% butterfly connections (dashed = negated)
\foreach \r [evaluate=\r as \m using {int(7-\r)}] in {0,...,3}{
  \draw (px\r)--(b\r);
  \draw (px\m)--(b\r);
}
\foreach \r [evaluate=\r as \m using {int(7-\r)}] in {4,...,7}{
  \draw (px\m)--(b\r);
  \draw (px\r)--(b\r) node[pos=0.85,below=0.3pt,font=\scriptsize]{$-$};
}
% reorder J4 (difference wires restored to natural order)
\foreach \r [evaluate=\r as \yy using {int(\r-4)},
             evaluate=\r as \lab using {int(7-\r)}] in {4,...,7}{
  \node[port] (c\r) at (4.9,\yy) {};
  \node[above=0.5pt] at (4.9,\yy) {\scriptsize $v_{\lab}$};
  \draw (b\r)--(c\r);
  \draw (c\r)--(5.3,\yy);
}
\foreach \r in {0,...,3}{ \draw (b\r)--(5.3,{7-\r}); }
% fixed fast cores
\draw[fill=blue!6] (5.3,3.55) rectangle (7.3,7.45);
\node[align=center] at (6.3,5.5)
  {$\C^{\mathrm{II}}_4$\\[-1pt]{\scriptsize fast DCT-II}};
\draw[fill=blue!6] (5.3,-0.45) rectangle (7.3,3.45);
\node[align=center] at (6.3,1.5)
  {$\C^{\mathrm{IV}}_4$\\[-1pt]{\scriptsize fast DCT-IV}};
% corrections
\foreach \yy in {0,...,7}{ \draw (7.3,\yy)--(7.7,\yy); }
\draw[fill=orange!12] (7.7,3.55) rectangle (9.6,7.45);
\node at (8.65,5.5) {$\Q^{(s)}_4(\rho)$};
\draw[fill=orange!12] (7.7,-0.45) rectangle (9.6,3.45);
\node at (8.65,1.5) {$\Q^{(a)}_4(\rho)$};
% outputs
\foreach \yy/\lab in {7/{y_0\ (n=0)},6/{y_2\ (n=2)},5/{y_4\ (n=4)},
  4/{y_6\ (n=6)},3/{y_1\ (n=1)},2/{y_3\ (n=3)},1/{y_5\ (n=5)},0/{y_7\ (n=7)}}{
  \node[port] (o\yy) at (10.5,\yy) {};
  \node[right] at (10.5,\yy) {$\lab$};
  \draw (9.6,\yy)--(o\yy);
}
% stage labels
\node at (1.55,8.3) {\itshape symmetric butterfly};
\node at (4.15,8.3) {\itshape reorder $\J_4$};
\node at (6.3,8.3) {\itshape fixed fast cores};
\node at (8.65,8.3) {\itshape corrections};
% footnote
\node at (5.3,-1.5)
  {\footnotesize $\Q^{(s)}_4=\Q^{(a)}_4=\I_4$ at $\rho=1$\quad
   (Chen's DCT-II factorization)};
\end{tikzpicture}
\caption{Signal-flow graph of the exact 8-point AR(1) KLT
\eqref{eq:factor}: fully symmetric butterfly, reorder stage $\J_4$
restoring the difference coordinates to natural order at the input of
the DCT-IV core, fixed fast cores $\C^{\mathrm{II}}_4$ and
$\C^{\mathrm{IV}}_4$, and the $\rho$-dependent corrections
$\Q^{(s)}_4,\Q^{(a)}_4$ of \eqref{eq:Qdef}. A $-$ sign marks a
negated branch. In the limit $\rho\to1$ the corrections equal $\I_4$
and the graph is Chen's DCT-II factorization.}
\label{fig:flow8}
\end{figure}

The entire design stage---everything computed offline from
$\rho$---still reduces to exact expressions in radicals, and the
corrections can be written down directly. The core eigenvalues
\eqref{eq:lambdas} at $M=4$ are
\begin{equation}
\Lambda^{\mathrm{II}}_4=\mathrm{diag}\big(0,\;2-\sqrt2,\;2,\;2+\sqrt2\big),
\qquad
\Lambda^{\mathrm{IV}}_4=\mathrm{diag}\big(2\mp\textstyle\sqrt{2\pm\sqrt2}\big),
\label{eq:lamvals}
\end{equation}
the latter listing
$2-\sqrt{2+\sqrt2},\,2-\sqrt{2-\sqrt2},\,2+\sqrt{2-\sqrt2},\,
2+\sqrt{2+\sqrt2}$ in ascending order, so the branch diagonals
$d_j=\rho\lambda_j+(1-\rho)^{2}$ are radicals in $\rho$. The
coupling weights need no separate table: from \eqref{eq:dctdefs},
\begin{equation}
\big(q^{(x)}_j\big)^{2}=\beta_j^{2}\,\frac{4-\lambda_j}{8},
\label{eq:qweights}
\end{equation}
with $\beta_j\equiv1$ for the antisymmetric branch---the couplings
are determined by the core eigenvalues themselves. The correction
eigenvalues are the roots of the quartic characteristic polynomial
of the tridiagonal blocks \eqref{eq:folds},
\begin{equation}
\begin{aligned}
p(\mu;\rho)=\mu^{4}
&-\big(4-\rho+3\rho^{2}\big)\mu^{3}
+\big(6-3\rho+6\rho^{2}-2\rho^{3}+3\rho^{4}\big)\mu^{2}\\
&-\big(4-3\rho+3\rho^{2}-2\rho^{3}+2\rho^{4}-\rho^{5}+\rho^{6}\big)\mu
+(1-\rho),
\end{aligned}
\label{eq:quartic}
\end{equation}
evaluated at $+\rho$ for the symmetric branch and at $-\rho$ for the
antisymmetric one (the sign substitution of Sec.~\ref{sec:n4});
its four roots $\mu_0<\dots<\mu_3$ are radicals in $\rho$ by the
classical quartic formulas (cf.\ Sec.~\ref{sec:secular}). The
entries then follow from \eqref{eq:eigvecs} in one line,
\begin{equation}
\big[\Q^{(x)}_4\big]_{ij}
=\frac{q^{(x)}_j}{d_j-\mu_i}
\Bigg(\sum_{k=0}^{3}\frac{\big(q^{(x)}_k\big)^{2}}{(d_k-\mu_i)^{2}}
\Bigg)^{\!-1/2},
\label{eq:Qentries}
\end{equation}
up to the overall sign of each row. QR elimination then delivers
the Givens decomposition of Sec.~\ref{sec:exact-impl} explicitly:
each correction is a product of six plane rotations,
\begin{equation}
\Q^{(x)}_4 \;=\; G_1G_2G_3G_4G_5G_6\,\mathbf{S}_x,
\qquad x\in\{s,a\},
\label{eq:givens}
\end{equation}
with $\mathbf{S}_x$ a diagonal of signs $\pm1$ and each $G_k$ a
Givens rotation: the $4\times4$ identity with the rotation
\eqref{eq:rot} embedded in a single coordinate plane $(p,q)$,
$p<q$. For example, in the plane $(p,q)=(1,3)$,
\begin{equation}
G_k=\begin{pmatrix}
1 & & & \\
 & \cos\theta_k & & -\sin\theta_k\\
 & & 1 & \\
 & \sin\theta_k & & \cos\theta_k
\end{pmatrix},
\label{eq:givdef}
\end{equation}
so that coordinates $1$ and $3$ are rotated by $\theta_k$ and the
others pass through unchanged. The planes are
$(p,q)=(0,1)$, $(0,2)$, $(0,3)$, $(1,2)$, $(1,3)$, $(2,3)$ for
$k=1,\dots,6$, and
the angles $\theta_k$ are produced by QR elimination of
$\Q^{(x)}_4$: each cosine--sine pair is a ratio of the entries
\eqref{eq:Qentries} and a square root thereof---again radicals in
$\rho$ (Fig.~\ref{fig:givens}). No iterative
eigensolver enters at any stage, and none of the asymptotic
machinery of Sec.~\ref{sec:largeN} is needed at these lengths. The
rotations are, moreover, small at high correlation
($\|\Q^{(x)}_4-\I_4\|_{\max}\le0.053$ already at $\rho=0.9$), which
keeps the chains \eqref{eq:givens} lifting- and integer-friendly.

\begin{figure}[t]
\centering
\begin{tikzpicture}[x=1cm,y=1.0cm,font=\small,
  port/.style={circle,fill,inner sep=1.1pt}]
% wires
\foreach \i in {0,...,3}{
  \node[left] at (0,{3-\i}) {$z_{\i}$};
  \node[port] (in\i) at (0.08,{3-\i}) {};
  \draw (in\i) -- (9.2,{3-\i});
  \node[port] at (9.2,{3-\i}) {};
  \node[right] at (9.2,{3-\i}) {$\bigl[\Q^{(x)}_4 z\bigr]_{\i}$};
}
% sign diagonal
\draw[fill=blue!6] (0.7,-0.45) rectangle (1.7,3.45);
\node at (1.2,1.5) {$\mathbf{S}_x$};
% six plane rotations, applied left to right: G6,...,G1;
% each is a butterfly: rising diagonal carries -sin (marked $-$)
\foreach \x/\k/\i/\j in {2.6/6/2/3, 3.7/5/1/3, 4.8/4/1/2,
                         5.9/3/0/3, 7.0/2/0/2, 8.1/1/0/1}{
  \draw ({\x-0.35},{3-\i}) -- ({\x+0.35},{3-\j});
  \draw ({\x-0.35},{3-\j}) -- ({\x+0.35},{3-\i})
    node[pos=0.26,right=-1pt,font=\scriptsize]{$-$};
  \node[port] at ({\x-0.35},{3-\i}) {};
  \node[port] at ({\x-0.35},{3-\j}) {};
  \node[port] at ({\x+0.35},{3-\i}) {};
  \node[port] at ({\x+0.35},{3-\j}) {};
  \node[above=1.5pt] at (\x,{3-\i}) {\scriptsize $G_{\k}$};
}
% stage labels
\node at (1.2,3.95) {\itshape signs};
\node at (5.35,3.95) {\itshape plane rotations (applied in order $G_6,\dots,G_1$)};
\end{tikzpicture}
\caption{Signal-flow graph of the six-rotation implementation
\eqref{eq:givens} of a correction $\Q^{(x)}_4(\rho)$,
$x\in\{s,a\}$: the $\pm1$ sign diagonal $\mathbf{S}_x$ followed by
the plane rotations $G_6,\dots,G_1$ of
\eqref{eq:givens}--\eqref{eq:givdef}, acting in the coordinate
pairs $(2,3)$, $(1,3)$, $(1,2)$, $(0,3)$, $(0,2)$, $(0,1)$. Each
rotation is drawn as a butterfly between the two wires it couples:
the through branches carry $\cos\theta_k$ and the crossing branches
$\pm\sin\theta_k$, the negated branch marked by a $-$ sign,
exactly as in
\eqref{eq:givdef}; wires crossed by a diagonal without a terminal
dot are unaffected. All cosine--sine pairs are computed offline, in
radicals in $\rho$. As $\rho\to1$ every angle tends to zero and
each butterfly degenerates to straight-through wires.}
\label{fig:givens}
\end{figure}
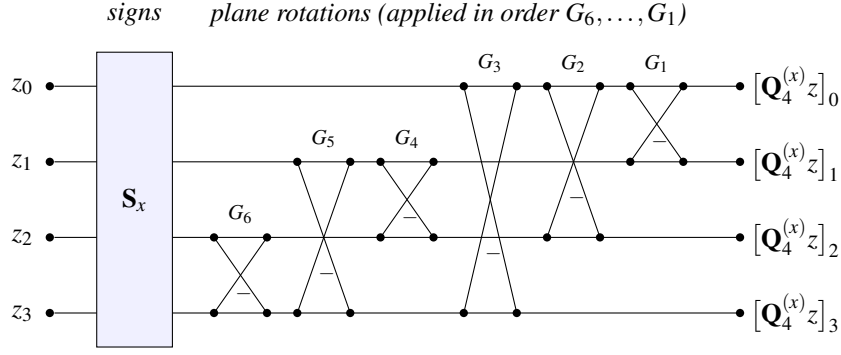

Table~\ref{tab:ops} collects the counts at $N=8$, with fixed DCT-II
algorithms for reference. Exact $\rho$-adaptivity costs about twice
Chen's algorithm: $32$ versus $16$ multiplications, or $24$ in
scaled form. At $N=16$ the same constructions give $256$ (dense),
$128$ (merged), and $112$ (scaled) multiplications against $44$ for
Chen---roughly a factor of three---and the gap widens like
$N/\log N$ as long as the KLT is computed exactly, since the exact
count is $N^{2}/2$ against the DCT's $O(N\log N)$. It contracts again at
large $N$: with the compressed corrections of
Sec.~\ref{sec:largeN}, the total returns to $O(N\log N)$ plus a
linear term and the overhead ratio decays toward one. The overhead of computing the full KLT is thus largest at
intermediate block sizes and mildest at the short lengths treated
above and in the asymptotic regime. One boundary is worth recording:
at $N=16$ the correction eigenvalues become roots of degree-eight
characteristic polynomials, no longer expressible in radicals, so
the design stage beyond $N=8$ uses the root-finding of
Sec.~\ref{sec:largeN}; the runtime realizations above are
unaffected.

\begin{table}[t]
\vspace{10pt}
\caption{Operation counts of exact realizations of the $8$-point
AR(1) KLT; all design constants are precomputed offline from $\rho$.
Fixed DCT-II algorithms are shown for reference. Rotation chains are
counted in three-multiplication form.}
\label{tab:ops}
\centering
\begin{tabular}{lcc}
\toprule
Realization & mult. & add. \\
\midrule
Dense $\W_8$ & 64 & 56 \\
Chen cores $+$ dense $\Q^{(s)}_4,\Q^{(a)}_4$ (Fig.~\ref{fig:flow8}) & 48 & 50 \\
Chen cores $+$ rotation chains \eqref{eq:givens} & 52 & 62 \\
Butterfly $+$ merged dense branches & 32 & 32 \\
Scaled merged (diagonal absorbed by quantizer) & 24 & 32 \\
\midrule
DCT-II, Chen \emph{et al.}~\cite{chen1977} & 16 & 26 \\
DCT-II, Loeffler \emph{et al.}~\cite{loeffler1989} & 11 & 29 \\
\bottomrule
\end{tabular}
\end{table}

\section{Large $N$: hierarchical $O(N\log N)$ application}
\label{sec:largeN}

At large $N$, the factorization \eqref{eq:factor} yields fast,
accuracy-controlled application of the \emph{exact} KLT. The cores
$\C^{\mathrm{II}}_M,\C^{\mathrm{IV}}_M$ cost $O(N\log N)$ by any of
the standard fast DCT
factorizations~\cite{chen1977,britanak2006}. Everything else is the
two corrections, and their eigenproblem \eqref{eq:Gdefs}, with
solution \eqref{eq:secfun}--\eqref{eq:eigvecs}, is not exotic: a
diagonal-plus-rank-one matrix is precisely the pivotal object of the
divide-and-conquer eigensolvers of Cuppen~\cite{cuppen1981} and
Dongarra and Sorensen~\cite{dongarra1987}, and a mature computational
toolkit surrounds it. The roots of \eqref{eq:secfun} (the ``secular equation'' of those
solvers) are found offline by their fast, guaranteed-convergent
iterations. At run time, multiplication by the
eigenvector matrix \eqref{eq:eigvecs} need not cost $O(M^{2})$: by
\eqref{eq:eigvecs} its entries have the Cauchy form
$\gamma_i\,q_j/(d_j-\mu_i)$---$\gamma_i$ being the row
normalizations---with interlaced pole sequences, and such
matrices can be applied in near-linear time by the fast multipole
method~\cite{greengard1987}, as done in the stabilized
divide-and-conquer algorithm of Gu and Eisenstat~\cite{gu1995}, or
through hierarchically semiseparable (HSS) compressed
representations~\cite{xia2010}.

At $M=4$ these asymptotic devices offer nothing over the dense or
Givens forms of Sec.~\ref{sec:exact-impl}; their significance
emerges at large block sizes. There, the off-diagonal blocks of the
corrections are hierarchically low-rank: numerically, the
$\varepsilon$-ranks are 8 at tolerance $10^{-6}$ (12 at $10^{-10}$)
for $N=256$, growing to only 9 (13) at $N=512$---essentially
$N$-independent, as expected from the interlaced-pole structure.
Fast-multipole~\cite{greengard1987,gu1995} or
HSS-compressed~\cite{xia2010} application of the corrections
therefore runs in $O(N\log N)$ for a prescribed numerical tolerance.
All factors are precomputed offline from the known $\rho$, and the
resulting transform applies the exact KLT to that tolerance. In
total, $T(N)=N+2F(N/2)+O(rN)$ operations, where $F$ is the cost of
the chosen fast core and $r$ the compression rank: $O(N\log N)$
with a linear correction term. Two remarks complete the complexity
picture of Sec.~\ref{sec:opcounts}. The direct form pays the
correction rank once, at the top level, whereas the recursive
factorization of Ref.~[\citenum{reznik2026tsp}], applied with the
same compression at every level, costs $O(rN\log N)$---the same
order in $N$, with $r$ multiplying the logarithm. And relative to a
pure fast DCT, the overhead of computing the exact KLT behaves like
$1+O(r/\log N)$, decaying with $N$: largest at intermediate block
sizes, it vanishes in both limits.

\section{Conclusions}
\label{sec:conclusion}

The DCT-II was introduced as a fast approximation to the AR(1)
KLT~\cite{anr1974}, and is known to be accurate only as $\rho$
approaches~1~\cite{clarke1981}. It was recently shown that the exact
KLT possesses fast structure of its own: it factorizes recursively
through its own half-order copies, at every $\rho$, in $O(N\log N)$
operations~\cite{reznik2026tsp}. This paper established the
complementary, direct form of that structure. For every
$\rho\in(0,1)$ and even $N$, the exact KLT is Chen's fast DCT-II
graph---the butterfly stage followed by the fixed half-size DCT-II
and DCT-IV cores---completed by two orthogonal corrections, defined
as diagonalizers of diagonal-plus-rank-one matrices localized at the
block boundary. The corrections carry the entire $\rho$-dependence
and are unavoidable: no interior value of $\rho$ reduces the KLT to
any fixed uniform-grid transform. The trigonometric part of the
algorithm is unmodified DCT machinery, reusable verbatim, and Chen's
classical algorithm is the $\rho\to1$ endpoint, where the
corrections reduce to identities.

Beyond the factorization itself, the paper made the corrections
concrete and priced them. At $N=4$ each correction is a single plane
rotation, and the merged transform is the fast 4-point DCT-II graph
with two $\rho$-tunable closed-form angles. At $N=8$ the entire
design is written in radicals: the coupling weights are determined
by the core eigenvalues, the correction eigenvalues by one explicit
quartic serving both parity branches under $\rho\mapsto-\rho$, the
entries by a one-line eigenvector formula, and each correction
factors into six explicit Givens rotations. $N=8$ is also a
boundary: beyond it the characteristic polynomials exceed degree
four, the design stage becomes numerical, and the runtime
realizations persist unchanged. In operation counts, the
factorization doubles as an accounting of the price of exact
$\rho$-adaptivity: a single extra multiplication at $N=4$ (six
against five), about a factor of two at $N=8$ ($32$ merged or
$24$ scaled, versus $16$ for Chen), a gap growing like $N/\log N$ at
intermediate sizes, and a decay back toward parity at large $N$,
where fast-multipole or hierarchically compressed application of the
Cauchy-structured corrections computes the exact KLT in $O(N\log N)$
operations to prescribed accuracy---the direct form paying the
compression rank once, against once per level for the recursive
form.

\end{document}